\documentclass[12pt,a4paper]{article}
\usepackage{graphicx} 
\usepackage{titlesec}
\usepackage{url}

\usepackage{setspace}
\usepackage{fancyhdr}
\usepackage{lastpage}
\usepackage{zhnumber}
\usepackage{amscd}
\usepackage{authblk}

\usepackage{amsmath}
\usepackage{amsthm}
\numberwithin{equation}{section}

\usepackage[utf8]{inputenc}
\usepackage[margin=1in]{geometry}
\usepackage[titletoc,title]{appendix}
\usepackage{graphicx}
\usepackage{mwe}
\usepackage{diagbox}

\usepackage{amsmath,amsfonts,amssymb,mathtools}
\usepackage{mathrsfs}

\newcounter{counter}
\usecounter{counter}
\addtocounter{counter}{1}

\newtheorem{theorem}{Theorem}[counter]
\newtheorem{corollary}[theorem]{Corollary}

\newtheorem{lemma}[theorem]{Lemma}
\newtheorem{example}[theorem]{Example}

\newtheorem{definition}{Definition}[counter]
\newtheorem*{remark}{Remark}
\numberwithin{equation}{counter}

\usepackage{graphicx,float}
\usepackage{subfigure}

\usepackage[ruled,vlined]{algorithm2e}
\usepackage{algorithmic}

\usepackage{enumerate}

\usepackage{hyperref} 
\usepackage{cleveref}
\hypersetup{colorlinks=true, urlcolor=black, linkcolor= blue,hypertex = true, anchorcolor = blue, citecolor =  blue} 

\titleformat{\section}{\centering\large}{\thesection}{1em}{} 
\titleformat{\subsection}{}{\thesubsection}{1em}{} 

\title{Bounds on two-distance sets in Euclidean space and Unit Sphere}

\fancypagestyle{firstpage}{%
    \fancyhf{} 

    \fancyfoot[L]{
        \footnotesize
        \textit{E-mail address}: $^1$weiqunc493@gmail.com, $^2$whyu@math.ncu.edu.tw 
    }
}

\author[1]{Wei-Chun Chen}
\author[2]{Wei-Hsuan Yu}
\affil[1]{\small National Changhua Senior High School, Changhua 50057, Taiwan.}
\affil[2]{\small Department of Mathematics, National Central University, Taoyuan 32001, Taiwan.}
\date{\small\today}

\begin{document}

\maketitle

\pagenumbering{arabic}

\thispagestyle{firstpage} 

\section*{Abstract}
We establish upper bounds for the size of two-distance sets in Euclidean space and spherical two-distance sets. The main recipe for obtaining upper bounds is the spectral method. We construct Seidel matrices to encode the distance relations and apply eigenvalue analysis to obtain explicit bounds.

For Euclidean space, we have the upper bounds for the cardinality $n$ of a two-distance set. 
\[
    n \le \dfrac{(d+1)\left(\left(\frac{1+\delta^2}{1-\delta^2}\right)^2 - 1\right)}{\left(\frac{1+\delta^2}{1-\delta^2}\right)^2-(d+1)}+1.
\]
if the two distances are $1$ and $\delta$ in $\mathbb{R}^d$.

For spherical two-distance sets with $n$ points and inner products $a, b$ on $\mathbb{S}^{d-1}$, we will have the following:

\[
\begin{cases}
        n \le \dfrac{d\left(\left(\dfrac{a+b-2}{b-a}\right)^2-1\right)}{\left(\dfrac{a+b-2}{b-a}\right)^2-d}, &a+b \ge 0;\\
        n \le \dfrac{(d+1)\left(\left(\dfrac{a+b-2}{b-a}\right)^2-1\right)}{\left(\dfrac{a+b-2}{b-a}\right)^2-(d+1)}, &a+b < 0.
\end{cases}
\]
Notice that the second bound (for $a+b < 0$) is the same as the relative bound for the equiangular lines in one higher dimension. 
\section*{1. Introduction}

Let $\mathbb{R}^d$ be the $d$-dimensional Euclidean space. A set $X$ in $\mathbb{R}^d$ is called a \emph{two-distance set} if the distance between any pair of points in $X$ takes only one of two possible values. The fundamental problem in this area is to determine the maximal cardinality of a two-distance set in $\mathbb{R}^d$. Let $g(d)$ denote the maximum cardinality of  a two-distance set in $\mathbb{R}^d$.

For the general dimension $d$, a well-known construction provides two-distance sets with $\binom{d+1}{2}$ points in $\mathbb{R}^d$, consisting of the midpoints of the edges of a regular simplex. For the upper bounds, Bannai, Bannai, and Stanton \cite{bannai1983upper} proved that $g(d) \leq \binom{d+2}{2}$. Therefore, in general people only know that $\binom{d+1}{2} \leq g(d) \leq \binom{d+2}{2}$. In this paper,  we show that if the ratio of the two distances are given, we can have smaller upper bounds for the size of a two-distance set. 

Larman, Rogers, and Seidel \cite{larman1977two} proved that any two-distance set $X$ in $\mathbb{R}^d$ with $|X| > 2d + 3$ have squared distance ratio $\delta^2 = (k-1)/k$ for some integer $k$ satisfying
\[
2 \le k \le \frac{1+\sqrt{2d}}{2}.
\]
The condition $|X| > 2d + 3$ was later improved to $|X| > 2d + 1$  by Neumaier \cite{neumaier1981distance}, who also showed the existence of $(2d+1)$-point two-distance sets whose $\delta^2$ cannot be expressed in the form $(k-1)/k$, for some $ k \in \mathbb N$.

The maximum cardinalities of two-distance sets were determined for $d\le 8$ \cite{croft19639,kelly1947elementary,LISONEK1997318}.

\begin{table}[h]
    \centering
    \begin{tabular}{c|cccccccc}
        $d$    & 1 & 2 & 3 & 4 & 5 & 6 & 7 & 8 \\ \hline
        $g(d)$ & 2 & 5 & 6 & 9 & 10 & 12 & 29 & 45 \\
    \end{tabular}
    \caption{The maximum cardinalities of two-distance sets for $d\le 8$}
    \label{tab:placeholder}
\end{table}

A two-distance set $X$ is called \emph{spherical} if it lies on the unit sphere $\mathbb{S}^{d-1}$. For spherical two-distance sets, Delsarte, Goethals and Seidel \cite{delsarte1991spherical} proved that the cardinality of a spherical two-distance set $S$ is bounded above by $\frac{1}{2}d(d+3)$ and the equality holds for dimensions $d = 2, 6, 22$. Glazyrin and Yu \cite{glazyrin2018upper} determined the maximum cardinalities of spherical two-distance sets in $\mathbb R^d$, $M(d)$, for all dimensions with possible exceptions $d=(2k+1)^2-3$ where $k \in \mathbb{N}$:
\[
M(d) = \frac{d(d+1)}{2}, \quad \forall d \geq 7 \text{ except } d = (2k+1)^2-3, \; k \in \mathbb{N}.
\]


\subsection*{\textit{1.1. Summary of results}}
We begin by considering the Cayley-Menger matrix associated with a two-distance set, and then construct a corresponding Seidel matrix by appropriately adding a carefully designed auxiliary matrix. We then determine the spectrum of this Seidel matrix using Weyl's inequality. Finally, we exploit the structural properties of Seidel matrices combined with the Cauchy-Schwarz inequality to establish the following upper bounds. \\

\textbf{Theorem A.} \textit{Let $X$ be an $n$-point two-distance set in $\mathbb{R}^d$ with distances $1$ ans $\delta$. Then}
\[
    n \leq \dfrac{(d+1)\left(\left(\dfrac{1+\delta^2}{1-\delta^2}\right)^2 - 1\right)}{\left(\dfrac{1+\delta^2}{1-\delta^2}\right)^2-(d+1)}+1.
\]

Using spectral properties, we establish that certain parameters is odd integers and derive bounds for all possible cases, reproducing the result of Larman, Rogers, and Seidel \cite{larman1977two} through a similar approach to that of Lemmens and Seidel \cite{LEMMENS1973494}.

For spherical two-distance sets, we apply the same framework starting with Gram matrices. The sign of the sum of inner products $a+b$ affects the spectral structure, leading to two distinct cases.

\textbf{Theorem B.} \textit{Let $X$ be an $n$-point spherical two-distance set in $\mathbb{S}^{d-1}$ with inner products $a, b$. If $a+b \geq 0$, then}
\[
    n \leq \dfrac{d\left(\left(\dfrac{a+b-2}{b-a}\right)^2-1\right)}{\left(\dfrac{a+b-2}{b-a}\right)^2-d}.
\]

This result extends the classical bound for equiangular line systems \cite{LEMMENS1973494} to the more general case where $a+b \geq 0$, with the equiangular case corresponding to $a+b = 0$.

\textbf{Theorem C.} \textit{There exists a bijective correspondence between $n$ points equiangular line systems in $\mathbb{R}^{d+1}$ with common angle $\alpha$ and families of $n$ points spherical two-distance sets in $\mathbb{S}^{d-1}$ with inner products $a, b$ satisfying $a+b < 0$ and $\frac{a-b}{a+b-2} = \alpha$.}

Similar integrality conditions hold for spherical cases. When $a+b < 0$, we establish a connection to equiangular line systems, providing a new spectral proof of classical results from Delsarte, Goethals, and Seidel \cite{delsarte1991spherical}.

Finally, we show that these upper bounds are attainable if the corresponding Seidel matrices have exactly two distinct eigenvalues, which occurs precisely when there exists an associated Equiangular Tight Frame. This provides a complete characterization of when our bounds are tight.

\subsection*{\textit{1.2. Outline}}

The objective of this research is to improve upper bounds on $|X|$ by exploiting the spectral properties of associated matrices when the ratio of the distances is specified.

 In Section 2, we transform the two-distance set into a Seidel matrices, which encode the distance relations and enable spectral analysis. In Section 3, we derive upper bounds on the cardinality of two-distance sets using spectral techniques and matrix theory. In Section 4, we focus on spherical two-distance sets, encode the distance relations and enable spectral analysis. In Section 5, we discuss bounds on spherical two-distance sets when the sum of inner products is nonnegative. In Section 6, we analyze the case when the sum of inner products is negative. In Section 7, we establish the connection between the existence of ETFs (Equiangular Tight Frames) and the attainability of the upper bounds derived in the preceding sections.

\section*{2. Construction of Seidel matrix}
\stepcounter{counter}

Lisoněk \cite{LISONEK1997318} introduce the following theorem.

\begin{theorem}[\cite{LISONEK1997318}]
 Let $(c_{ij})$ be a real symmetric $n\times n$ matrix with zero diagonal.  There exist $n$ points $P_1, P_2 , \dots , P_{n} \in \mathbb R^d$ such that $c_{ij} = \|P_i - P_j \|$ if and only if the matrix ($i,j = 1, \dots , n-1$)
 \[
 M_{n-1} = (c_{i,n}^2 + c_{j,n}^2 - c_{i,j}^2)_{ij} 
 \]
 is positive semidefinite and has rank at most $d$.
\end{theorem}

    $M$ is also called the Cayley-Menger matrix. For any two-distance set, without loss of generality, we can rescale the set so that the two distinct distances are $1$ and $\delta > 1$. Suppose the distances from $P_1, P_2,\dots ,P_h$ to $P_{n}$ are $1$, and those from the remaining points to $P_{n}$ are $\delta$. Let $M$ be defined as follows:
\[
M =  \begin{bmatrix}
        2 &    & 2-\delta_i^2 &  &  &   &  &  \\
         &   2  &  &                   &  & \delta_i^2  &  &  \\
        2-\delta_i^2   & & \ddots &      &  & \\
           &   &  & 2                  &  &   &  &  \\
           &  &   &  & 2\delta^2 &    & 2\delta^2-\delta_i^2 &  \\
           &  \delta_i^2&   &  &  &   2\delta^2  &  &  \\
           &  & &  & 2\delta^2-\delta_i^2 & & \ddots & \\
           &  &   &  &    &   &  & 2\delta^2\\
    \end{bmatrix}_{(n-1)},
\]
    where $\delta_i \in \{ 1 ,  \delta\}$.

To facilitate analysis of two-distance sets, we transform the Cayley-Menger matrix into a form Seidel matrix by constructing a block matrix $D$.
    
\begin{definition} Let $X$ be a two-distance set with $n$ points in $\mathbb R^d$, where the two distances are $1$ and $\delta$. Let $M = M(X)$ be the Cayley-Menger matrix. Define the Seidel matrix $ S$ by
\[
 S_{n-1} = \frac{2M_{n-1} +  D_{n-1}  - (1 + \delta^2) \cdot I_{n-1} } {\delta^2-1}
\] 
where
\[
D_{n-1} = \begin{bmatrix}
    (-3 + \delta^2) J_h & -(1 + \delta^2) J_{h,n-h-1}  \\
    -(1 + \delta^2) J_{n-h-1,k} & (1 -3 \delta^2) J_{n-h-1} \\
\end{bmatrix}_{n-1}.
\]
where $J_{p, q}$ is a $p \times q$ all-ones matrix.
\end{definition}

\begin{remark}
    Notice that upper left block of $2M$ has non-diagnal entries $2$ or $2(2-\delta^2)$. We subtract the middle value of them, and normalize it. So $D_{n-1}$ is the negative middle value.
\end{remark}

\begin{lemma}
    The matrix S defined above is indeed a Seidel matrix.
\end{lemma} 
\begin{proof}
    We analyze each block of the matrix to verify that S has the structure of a Seidel matrix, having zeros on the diagonal and $\pm 1$ entries non-diagonal.
    
    First, the diagonal elements of the upper left $h\times h$ block are calculated as $(4+(-3+\delta^2)-(1+\delta^2))/(\delta^2-1)$, which yield $0$. The non-diagonal elements are found to be $(4-2\delta_i^2+(-3+\delta^2))/(\delta^2-1) = (1+\delta^2 -2\delta_i^2)/(\delta^2-1) = \pm 1$. A similar calculation applies to the lower right $(n-h-1)\times (n-h-1)$ block.
    
    Next, we consider the upper right $h \times (n-h-1)$ block. The elements in this block are $(2\delta^2_i-(1+\delta^2))/(\delta^2-1) = \pm 1$. A similar calculation holds for the lower left $(n-h-1)\times h$ block.
    
    Therefore, $S$ has zero diagonal elements and non-diagonal elements equal to $\pm 1$. Since $M$, $D$, and $I$ are all symmetric matrices, their linear combination $S$ is also symmetric. This confirms that $S$ has the structure of a Seidel matrix.
\end{proof}

    Next, we determine the spectrum of matrix $D$.

\begin{lemma}
    The spectrum of $D$ is $\{ [a_2] ,[0]^{n-3}, [a_1]\}$ with $a_2 < 0 < a_1$, where $i = 1, 2$
\begin{equation*}
\begin{aligned}
    a_i =&  -2(1-\delta^2)h+\left(\dfrac{1-3\delta^2}{2}\right)(n-1) \\
    &+ (-1)^{i+1}  \sqrt{2(n-1)h(1-\delta^2)(1+\delta^2) + (n-1)^2(\frac{1-3\delta^2}{2})^2}.
\end{aligned}
\end{equation*}
\end{lemma}
\begin{proof}
   We determine the spectrum of matrix $D$ by analyzing its eigenvalue structure.

    \textbf{Step 1:} Zero eigenvalues from block structure
    
    The spectrum of $J_h$ is $\{ [0]^{h-1}, [h]\}$. Let $\mathbf v_1, \mathbf v_2,\dots,\mathbf v_{h-1}$ be the eigenvectors of $J_h$ which correspond to the eigenvalue 0. Then the sum of the entries in $\mathbf v_i$ are $0$, where $i=1,\dots,h-1$. Define ($i=1,\dots,h-1$)
    \[
    \mathbf v_i' = \left[\begin{matrix}
        \mathbf v_i^{\mathsf T} &0 &0 &\dots &0
    \end{matrix}\right]^{\mathsf T} 
    \]
    be a $(n-1)\times1$ vector with $n-h-1$ $0$s at the end. ince $D\mathbf{v}_i' = 0 \cdot \mathbf{v}_i'$, each $\mathbf{v}_i'$ is an eigenvector of $D$ with eigenvalue $0$. 
    Similarly, the spectrum of $J_{n-h-1}$ is $\{[0]^{n-h-2}, [n-h-1]\}$. Let $\mathbf{v}_h, \mathbf{v}_{h+1}, \ldots, \mathbf{v}_{n-3}$ be the eigenvectors of $J_{n-h-1}$ corresponding to eigenvalue 0. We extend these by defining ($j=h,\dots,n-3$)
    \[
    \mathbf v_j' = \left[\begin{matrix}
        0 &0 &\dots &0 &\mathbf v_j^{\mathsf T} 
    \end{matrix}\right]^{\mathsf T} 
    \]
    be a $(n-1)\times1$ vector with $h$ 0s at the beginning. Since $D\mathbf{v}_j' = 0 \cdot \mathbf{v}_j'$, each $\mathbf{v}_j'$ is an eigenvector of $D$ with eigenvalue $0$. Therefore, $D$ has eigenvalue $0$ with multiplicity $n-3$.
    
    \textbf{Step 2:} Non-zero eigenvalues

    To find the remaining eigenvalues, we consider eigenvectors of the form
    \[
\mathbf{u} = \left[\begin{matrix}
x & x & \ldots & x & 1 & 1 & \ldots & 1
\end{matrix}\right]^{\mathsf T},
\]
where there are $h$ entries equal to $x$ and $n-h-1$ entries equal to 1.
Computing $D\mathbf{u}$, we obtain
\[
     D\mathbf u = \begin{bmatrix}
        (-3+\delta^2)hx-(1+\delta^2)(n-h-1) \\
        (-3+\delta^2)hx-(1+\delta^2)(n-h-1) \\
         \vdots \\
        (-3+\delta^2)hx-(1+\delta^2)(n-h-1) \\
         -(1+\delta^2)hx+(1-3\delta^2)(n-h-1) \\
         -(1+\delta^2)hx+(1-3\delta^2)(n-h-1) \\
         \vdots \\
         -(1+\delta^2)hx+(1-3\delta^2)(n-h-1) 
    \end{bmatrix}
    \]
For $\mathbf{u}$ to be an eigenvector, we need $D\mathbf{u} = \lambda\mathbf{u}$, which requires:
\begin{align}
(-3+\delta^2)hx-(1+\delta^2)(n-h-1) &= \lambda x \\
-(1+\delta^2)hx+(1-3\delta^2)(n-h-1) &= \lambda
\end{align}
   From the consistency condition, we get:
    \begin{equation*}
    \begin{aligned}
    \dfrac{(-3+\delta^2)hx-(1+\delta^2)(n-h-1)}{-(1+\delta^2)hx+(1-3\delta^2)(n-h-1} = x
    \end{aligned}
    \end{equation*}
Solving for $x$:
    \[
    x = \frac{2(1+\delta^2)h+(1-3\delta^2)(n-1)\pm \sqrt{8(n-1)h(1+\delta^2)(1-\delta^2)+(n-1)^2(1-3\delta^2)^2}}{2h(1+\delta^2)}
    \]
The corresponding eigenvalues are:
    \begin{equation*}
    \begin{aligned}
      a_{1,2} = -(1+\delta^2)hx+(1-&3\delta^2)(n-h-1) =   -2(1-\delta^2)h+\left(\dfrac{1-3\delta^2}{2}\right)(n-1) \\
    &\pm \sqrt{2(n-1)h(1+\delta^2)(1-\delta^2)+(n-1)^2\left(\dfrac{1-3\delta^2}{2}\right)^2}
    \end{aligned}
    \end{equation*}
    
    \textbf{Step 3:} Sign analysis\\
    \begin{equation*}
    \begin{aligned}
        \text{ Let }p = -2(1-\delta^2)h+\left(\dfrac{1-3\delta^2}{2}\right)&(n-1), \\
        \text{ and }q &=\sqrt{2(n-1)h(1+\delta^2)(1-\delta^2)+(n-1)^2\left(\dfrac{1-3\delta^2}{2}\right)^2}.
    \end{aligned}
    \end{equation*}
    Then $p^2 - q^2$
    \begin{equation*}
    \begin{aligned}
    &=(-2(1-\delta^2)h+\left(\dfrac{1-3\delta^2}{2}\right)(n-1))^2 - 2(n-1)h(1+\delta^2)(1-\delta^2) - (n-1)^2\left(\dfrac{1-3\delta^2}{2}\right)^2\\
    & =  4(1-\delta^2)^2h^2 -2(n-1)h(1-\delta^2)(1-3\delta^2)- 2(n-1)h(1+\delta^2)(1-\delta^2) \\
    & = 4(1-\delta^2)^2h^2 -4(n-1)h(1-\delta^2)^2 \\
    & = -4(1-\delta^2)^2h(n-h-1) < 0
    \end{aligned}
    \end{equation*}
    Since $p^2 < q^2$, we have $|p| < q$, which implies $a_2 = p - q < 0 < p + q = a_1$. 
    
    Therefore, the spectrum of $D$ is $\{[a_1], [0]^{n-3}, [a_2]\}$ with $a_1 < 0 < a_2$. $\square$
\end{proof}    

In order to analyze the spectrum of $S$, we apply Weyl's inequality.

\begin{theorem}[Weyl's inequality\cite{franklin2000matrix}]
Let $M = N + R$, $N$, and $R$ be $n \times n$ symmetric matrices, with their respective eigenvalues in descending order $\lambda_1 \ge \lambda_2 \ge \dots \ge \lambda_n$.
Then the following inequalities hold:

\[
\lambda_i(N) + \lambda_n (R) \le \lambda_i(M) \le \lambda_i (N) + \lambda_1 (R)
\]
for $i= 1,\dots,n$.
More generally,
\[
\lambda_j(N) + \lambda_k (R) \le \lambda_i(M) \le \lambda_r (N) + \lambda_s (R)
\]
for $j+k- n\ge i \ge r+s-1$.
\end{theorem}

\begin{theorem}
    $ S = \frac{2M +  D  - (1 + \delta^2) \cdot I } {\delta^2-1}$ has the smallest eigenvalue $\lambda_{n-1}$ with multiplicity $1$ and second smallest eigenvalue $\frac{1+\delta^2}{1-\delta^2}$ with the multiplicity at least $n-d-3$.
\end{theorem}
\begin{proof}

Consider the eigenvalue of $W = 2M + D$. Let $\lambda_1 \ge \lambda_2 \ge \dots \ge \lambda_{n-1}$ denote the eigenvalues of matrices $2M$, $D$, and $W$ in descending order, respectively.

Since $2M$ is a positive semidefinite matrix with rank at most $n-d-1$, we have $\lambda_{n-1}(2M) =\lambda_{n-2}(2M) = \dots =\lambda_{d+1}(2M) = 0$.
\[
\lambda_{n-1}(D) < 0 ,\ \lambda_{n-2}(D) = \dots =\lambda_{2}(D) = 0 ,\ \lambda_{1}(D) > 0
\]
Applying Weyl's inequality,
\[
 \lambda_{n-1}(W) \le 0 
\]
\[
0 \le \lambda_{n-i}(W) \le 0
\]
\[
0 \le \lambda_{d+1}(W), \ 0 \le \lambda_{d}(W)
\]
\[
0 < \lambda_{n-j}(W) 
\]
for $i= 2,\dots,n-d-2 , j = n-d+1 , \dots , n-1$.

Therefore, the matrix $W = 2M + D$ has the following eigenvalue structure: $\lambda_{n-1}(W) \leq 0$, eigenvalues $\lambda_{n-2}(W) = \cdots = \lambda_{d+2}(W) = 0$ with multiplicity $n-d-3$, and positive eigenvalues $\lambda_{d+1}(W), \ldots, \lambda_1(W) \ge 0$.

Since ${S} = \frac{W - (1+\delta^2)I}{\delta^2-1}$ and $\delta^2 - 1 > 0$, the eigenvalues of ${S}$ are given by:
\[
\lambda_i({S}) = \frac{\lambda_i(W) - (1+\delta^2)}{\delta^2-1}
\]

This transformation preserves the ordering and multiplicities. The zero eigenvalues of $W$ become $\frac{-(1+\delta^2)}{\delta^2-1} = \frac{1+\delta^2}{1-\delta^2}$ in ${S}$ with multiplicity at least $n-d-3$. The smallest eigenvalue $\lambda_{n-1}(W) \leq 0$ becomes the smallest eigenvalue $\lambda_{n-1}({S})$ of ${S}$, completing the proof.
\end{proof}

\section*{3. Upper bound of two distance sets}
\stepcounter{counter}

The spectral properties established in the preceding section enable us to derive upper bounds on the size of two-distance sets.

\begin{theorem} \label{lem:RelativeBound}
Let $X$ be an $n$ points two-distance set in $\mathbb R^d$ with distances $1,\delta$.  Then 
\begin{equation}
    n \le \frac{(d+1)\left(\left(\dfrac{1+\delta^2}{1-\delta^2}\right)^2 - 1\right)}{\left(\dfrac{1+\delta^2}{1-\delta^2}\right)^2-(d+1)}+1.
\end{equation} 
\end{theorem}

\begin{proof}
Let $ S =  S(X)$ be the Seidel matrix of $X$ with smallest eigenvalue $\lambda_{n-1}$ with multiplicity $1$ and second smallest eigenvalue $-\gamma = (1+\delta^2)/(1-\delta^2) < 0$ with multiplicity $n-d-3$ and the others eigenvalues $\lambda_{d+1} ,  \lambda_{d} , \dots , \lambda_1$. By Cauchy-Schwarz inequality,
\begin{equation} \label{cauchy1}
    \left( \sum^{d+1}_{i=1} \lambda_i\right)^2 \le (d+1)\sum^{d+1}_{i=1} \lambda_i^2.
\end{equation}
Since $S$ is a Seidel matrix,
\begin{gather}
        \mathrm{tr}( S) = \lambda_{n-1} - (n-d-3)\gamma + \sum^{d+1}_{i=1} \lambda_i  = 0 ; \\
    \mathrm{tr}( S^2) = \lambda_{n-1}^2 + (n-d-3)\gamma^2 + \sum^{d+1}_{i=1} \lambda_i^2 = (n-1)(n-2).
\end{gather}

    Substitution to equation \ref{cauchy1}, we have $ (\lambda_{n-1} - (n-d-3)\gamma)^2 \le (d+1)((n-1)(n - 2) -  \lambda_{n-1}^2 - (n-d-3)\gamma^2)$. After simplified, 
\begin{equation}
\begin{aligned}
    (\gamma^2 - (d+1))(n-1&)^2 - (d+1)(\gamma^2-1)(n-1) \\ &-2\gamma(\lambda_{n-1}+\gamma)(n-1) + (d+2)(\lambda_{n-1} + \gamma)^2 \le 0    
\end{aligned}
\end{equation}
Since $\gamma > 0$ and $\lambda_{n-1} \le -\gamma$,
\begin{equation*} 
\begin{aligned}
    (\gamma^2 - (d+1))(n-1)^2 - (d+1)(\gamma^2-1)(n-1) &\le \\ 2\gamma(\lambda_{n-1}+\gamma)(n-1) - (d+2)(\lambda_{n-1} + \gamma)^2 &\le 0\\
\end{aligned}
\end{equation*}
To conclude, we have
\begin{equation}\label{last}
     (\gamma^2 - (d+1))(n-1) \le  (d+1)(\gamma^2-1).
\end{equation}

In the case of equality, we have  $\gamma = \sqrt{\frac{(n-2)(d+1)}{n-d-2}}$, equality in the Cauchy-Schwarz inequality, which implies that $\lambda_1 = \lambda_2 = \dots = \lambda_{d+1}$ and equality in \ref{last}, which implies that $\lambda_{n-1} = -\gamma$.
Since $\mathrm{tr} ( S) = 0$, we find that $ S$ has spectrum
\[
\left\{ \left[-\sqrt{\dfrac{(n-2)(d+1)}{n-d-2}}\right]^{n-d-2},\left[\sqrt{\dfrac{(n-2)(n-d-2)}{d+1}}\right]^{d+1} \right\}.
\]
\end{proof}

\begin{theorem}[\cite{LEMMENS1973494}] \label{thm:odd}
Let $X$ be an $n$ two-distance points set in $\mathbb R^d$ with distances $1,\delta$. If $n>2d+4$, $-\gamma = (1+\delta^2)/(1-\delta^2)$ is an odd integer.
\end{theorem} 
    Larman, Rogers, and Seidel \cite{larman1977two} claimed that if $n > 2d + 3$, the distance square $\delta^2$ equals $(k-1)/k$ for some integer $k$, which is equivalent to the condition $-(1+\delta^2)/(1-\delta^2) = 2k-1$. The condition $|X| > 2d + 3$ was improved to $n > 2d + 1$  by Neumaier \cite{neumaier1981distance}
    We adopt this notation and refer to such $k$ as the L.R.S. constant.
\begin{proof}
    Let $ S =  S(X)$ be the Seidel matrix of $X$ with smallest eigenvalue $\lambda_{n-1}$ and second smallest eigenvalue $-\gamma = (1+\delta^2)/(1-\delta^2)$ with multiplicity $n-d-3$. Since $ S$ is a $\{0,\pm1\}$-matrix, the eigenvalues are algebraic integers. Hence, every algebraic conjugate of $\gamma$ is also an eigenvalue of $ S$ with multiplicity $n-d-3$. If $n>2d+4$, since $1 +(n-d-3)+(n-d-3) = n-1+(n-2d-4) > n-1$, $ S$ cannot have more than one eigenvalue with multiplicity $n-d-3$. Therefore $\gamma$ is a rational number.

    Let $A = \frac 12 (J-I-S)$ and $\mathbf v_1, \mathbf v_2$ are eigenvectors of $ S$ satisfying $ S\mathbf v_i = \gamma\mathbf v_i$, where $i=1,2$. Let 
    \[
    \mathbf v_3 =  \mathbf v_1 \cdot \sum_{j=1}^{n-1} \mathbf v_{2j} - \mathbf v_2 \cdot \sum_{j=1}^{n-1} \mathbf v_{1j}.
    \]
    Since $J\mathbf v_3 = 0$, 
    \[
    A\mathbf v_3 = \frac 12 (J-I-S) \mathbf v_3 = \frac 12 (0-1+\gamma) \mathbf v_3.
    \]
    Hence $\frac 12 (-1+\gamma)$ is an eigenvalue of $A$. Since $A$ is integer matrix, $\frac 12 (-1+\gamma)$ is an algebraic integer. But $\frac 12 (-1+\gamma)$ is rational ($\gamma$ is rational), we have $\frac 12 (-1+\gamma)$ is integer, therefore $\gamma$ is an odd integer.
\end{proof}

\begin{definition}
    For an odd integer $\gamma$, define $g_\gamma(d)$ as the maximum cardinality of two-distance sets in $\mathbb R^d$ with distances $1,\delta$ such that $(1+\delta^2)/(1-\delta^2) = -\gamma$.
\end{definition}

\begin{theorem}[\cite{LEMMENS1973494}] \label{LemmenSeidel}Suppose $m, d$ are positive integer with $(2m+1)^2 > d > 3$. Then
\[
g(d) \le \max\left\{g_3(d), g_5(d),\dots, g_{2m+1}(d) , \frac{(d+1)4m(m+1)}{4m^2+4m-d}+1\right\}.
\]
\end{theorem}
\begin{proof}
    By Theorem \ref{thm:odd}, $(1+\delta^2)/(1-\delta^2) = \gamma$ is an odd integer. For a positive number $m$, if $\gamma > 2m+1$, since $\gamma > (2m+1)^2 > d$, by Lemma \ref{lem:RelativeBound}, we have
    \begin{equation*}
    \begin{aligned}
        g_{\gamma}(d)-1 \le \frac{(d+1)\left(\gamma^2 - 1\right)}{\gamma^2-(d+1)} &= d+1 +\frac{d(d+1)}{\gamma^2-(d+1)}\\
    &\le d+1 +\frac{d(d+1)}{(2m+1)^2-(d+1)} = \frac{(d+1)4m(m+1)}{4m^2+4m-d}.
    \end{aligned}
    \end{equation*}
\end{proof}

\newpage

\begin{table}[ht]
    \centering
    \begin{tabular}{c|cccc|c}
         \diagbox{dimension}{upper bound}{ $\ k$} & 2 & 3 & 4 & 5 & $g(d)$\\  \hline 
         $5$ & $17$ & 8  & 7  & 7  & 16 \\
          6  & $29$ & 10 & 9  & 8  & 27 \\  
         $7$ & $65$ & 12 & 10 & 9  & 29 \\ \hline
         $8$  & & $14$  &  11 & 11 & 45\\ 
          9   & & $17$  &  13 & 12 & 45- \\
         $10$ & & $19$  &  14 & 13 & 55- \\
         $11$ & & $23$  &  16 & 14 & 66- \\
         $12$ & & $27$  &  18 & 16 & 78- \\
         $13$ & & $31$  &  20 & 17 & 91- \\
         $14$ & & $37$  &  22 & 19 & 105- \\
         $15$ & & $43$  &  24 & 20 & 120- \\
         $16$ & & $52$  &  26 & 22 & 136- \\
         $17$ & & $62$  &  28 & 23 & 153- \\
         $18$ & & $77$  &  31 & 25 & 171- \\
         $19$ & & $97$  &  34 & 27 & 190- \\
         $20$ & & $127$ &  37 & 29 & 210- \\
         $21$ & & $177$ &  40 & 30 & 231- \\
         $22$ & & $277$ &  43 & 32 & 253- \\
         $23$ & & $577$ &  47 & 34 & 276- \\ \hline
         $24$ & & & $51$  & 36 & 300- \\
          25  & & & $55$  & 38 & 325- \\ 
         $26$ & & & $59$  & 41 & 351- \\
         $27$ & & & $65$  & 43 & 378- \\    
         $28$ & & & $70$  & 45 & 406- \\
         $29$ & & & $76$  & 48 & 435- \\
         $30$ & & & $83$  & 50 & 465- \\
         $31$ & & & $91$  & 53 & 496- \\
         $32$ & & & $100$ & 56 & 528- \\
         $33$ & & & $109$ & 58 & 561- \\
    \end{tabular}
    \caption{Upper Bounds on two-distance sets in $d$-dimensional Euclidean space given by theorem \ref{lem:RelativeBound}. For each dimension $d$, the upper bound applies for the L.R.S. constant $k = 2, 3, 4, 5$, where $-(1+\delta^2)/(1-\delta^2) = 2k-1$.}
    \label{tab:bounds}
\end{table}


\begin{example}
In Table \ref{tab:bounds}, for $\mathbb{R}^{15}$, the upper bound is 43 points for distance ratios $\delta \geq \frac{\sqrt{3}}{\sqrt{2}}$. However, we have constructed a configuration with 120 points for distance ratio $\delta = \sqrt{2}$ (the mid-point case). This indicates that the maximum two-distance set is achievable only for the specific distance ratio $\sqrt{2}$.
\end{example}

\section*{4. spherical two distance sets}
\stepcounter{counter}

Having established Bounds on two-distance sets in general Euclidean space, we now turn our attention to the special case of spherical two-distance sets. 

If a two-distance set $S$ lies in the unit sphere $\mathbb S^{d-1}$, then $S$ is called spherical two-distance set. In other words, $S$ is a set of unit vectors, there exist two distinct real numbers $a$ and $b$ with $-1 \le a < b < 1$, and inner products of distinct vectors in $S$ are either $a$ or $b$.

\begin{definition}
    For a spherical two-distance set with points $v_1, v_2 , \dots , v_n$ in $\mathbb S^{d-1}$, the Gram matrix is defined by $i, j = 1, \dots, n$ 
    \[
    G = (\langle v_i , v_j \rangle)_{ij}.
    \]
\end{definition}

\begin{definition} Let $X$ be an $n$ points spherical two-distance set in $\mathbb S^{d-1}$ with inner products $a,b$ ($a<b$). Let $G = G(X)$ be the Gram matrix. Define the Seidel matrix $S$ by
\[
S = \dfrac{(G-I)-\dfrac{a+b}{2}\cdot J + \dfrac{a+b}{2}\cdot I}{\dfrac{b-a}{2}}.
\]
\end{definition}

\begin{lemma}
    The matrix S defined above is indeed a Seidel matrix.
\end{lemma} 
\begin{proof}  
    First, we examine the diagonal elements. Since the diagonal of $G$ are all $1$, the the diagonal of $S$ are $(1-1-\frac{a+b}{2}+\frac{a+b}{2})/\frac{b-a}{2} = 0$.
    
    Next, we consider the non-diagonal elements. Since the non-diagonal elements of $G$ are either $a$ or $b$. If the element is $a$, the same position of $S$ will be $(a-\frac{a+b}{2})/\frac{b-a}{2} = \frac{a-b}{2}/\frac{b-a}{2} =-1$; If the element is $b$, the same position of $S$ will be $(b-\frac{a+b}{2})/\frac{b-a}{2} = \frac{b-a}{2}/\frac{b-a}{2} =1$. 
    
    Therefore, $S$ has zero diagonal elements and non-diagonal elements equal to $\pm 1$. Since $G$, $J$, and $I$ are all symmetric matrices, their linear combination $S$ is also symmetric. This confirms that $S$ has the structure of a Seidel matrix.
\end{proof}

Note that the Seidel matrix of spherical two-distance set is order $n$.

\begin{theorem} \label{thm:eigenvalues}
    Let $ S = \frac{2G - (a+b)\cdot J  + (a+b-2) \cdot I } {b-a}$.
    If $a+b \ge 0$, $S$ has the smallest eigenvalue $\lambda_{n}$ with multiplicity $1$ and second smallest eigenvalue $\frac{a+b}{b-a}$ with the multiplicity at least $n-d-1$; If $a+b < 0$, $S$ has the smallest eigenvalue $\frac{a+b-2}{b-a}$ with the multiplicity at least $n-d-1$.
\end{theorem}
\begin{proof}

Consider the eigenvalue of $N = 2G -(a+b)\cdot J$. Let $\lambda_1 \ge \lambda_2 \ge \dots \ge \lambda_n$ denote the eigenvalues of matrices $2G$, $-(a+b)\cdot J$, and $N$ in descending order, respectively.

If  $a+b \ge 0$, Since $2G$ is a positive semidefinite matrix with rank at most $n-d$, we have $\lambda_{n}(2G) =\lambda_{n-1}(2G) = \dots =\lambda_{d+1}(2G) = 0$ and 
\[
\lambda_{n}(-(a+b)\cdot J) =-(a+b)\cdot n ,\ \lambda_{n-1}(-(a+b)\cdot J) = \dots =\lambda_{1}(-(a+b)\cdot J) = 0 
\]
By Weyl's inequality, 
\[
 \lambda_{n}(N) \le 0,
\]
\[
0 \le \lambda_{n-i}(N) \le 0,
\]
\[
 0 \le \lambda_{d}(N),
\]
\[
0 < \lambda_{n-j}(N)
\]
for $i= 1,\dots,n-d-1 , j = n-d+1 , \dots , n-1$.

Therefore, the matrix $N = 2G -(a+b)\cdot J$ has the following eigenvalue structure: $\lambda_n(N) \leq 0$, eigenvalues $\lambda_{n-1}(N) = \cdots = \lambda_{d+1}(N) = 0$ with multiplicity $n-d-1$, and positive eigenvalues $\lambda_{d}(N), \ldots, \lambda_1(N) \ge 0$.

If  $a+b < 0$, we know that $\lambda_{n}(2G) =\lambda_{n-1}(2G) = \dots =\lambda_{d+1}(2G) = 0$ and 
\[ 
\lambda_{n}(-(a+b)\cdot J) = \dots =\lambda_{2}(-(a+b)\cdot J) = 0 , \ \lambda_{1}(-(a+b)\cdot J) = -(a+b)\cdot n
\]
By Weyl's inequality, 
\[
0 \le \lambda_{n-i}(N) \le 0,
\]
\[
 0 \le \lambda_{d+1}(N),
\]
\[
0 < \lambda_{n-j}(N)
\]
for $i= 0,\dots,n-d-2 , j = n-d , \dots , n-1$.

Therefore, the matrix $N = 2G -(a+b)\cdot J$ has the following eigenvalue structure: $\lambda_{n}(N) = \cdots = \lambda_{d+2}(N) = 0$ with multiplicity $n-d-1$, and positive eigenvalues $\lambda_{d+1}(N), \ldots, \lambda_1(N) \ge 0$.

Since ${S} = \frac{N+(a+b-2)\cdot I}{a-b}$ and $b-a > 0$, the eigenvalues of ${S}$ are given by:
\[
\lambda_i({S}) = \frac{\lambda_i(N) +(a+b-2)}{b-a}
\]

This transformation preserves the ordering and multiplicities. The zero eigenvalues of $N$ become $\frac{a+b-2}{b-a}$ in ${S}$ with multiplicity at least $n-d-1$. 

The smallest eigenvalue $\lambda_n(N) \leq 0$ becomes the smallest eigenvalue $\lambda_n({S})$ of ${S}$, completing the proof.
\end{proof}

Since the spectral structure differs depending on whether $a+b \geq 0$ or $a+b < 0$, we analyze these two cases separately in the following sections.

\section*{5. Upper bound of spherical two-distance sets when $a+b \ge 0$}
\stepcounter{counter}

\begin{theorem} \label{lem:SpericalRelativeBound}
Let $X$ be an $n$ points spherical two-distance set in $\mathbb S^{d-1}$ with inner products $a,b$.  If $a+b \ge 0$, then 
\begin{equation}
    n \le \dfrac{d\left(\left(\dfrac{a+b-2}{b-a}\right)^2-1\right)}{\left(\dfrac{a+b-2}{b-a}\right)^2-d}.
\end{equation} 
\end{theorem}

\begin{proof}

The proof follows similarly to that of Theorem \ref{lem:RelativeBound}. Let $-\gamma = (a+b-2)/(b-a)$ with multiplicity $n-d-1$. We have
\begin{equation} \label{eq:SphericalLast}
    (\gamma^2 - d)n \le  d(\gamma^2-1).
\end{equation}

In the case of equality, we have  $\gamma = \sqrt{\frac{d(n-1)}{n-d}}$, equality in the Cauchy-Schwarz inequality , which implies that $\lambda_1 = \lambda_2 = \dots = \lambda_{d+1}$ and equality in \ref{eq:SphericalLast}, which implies that $\lambda_n = -\gamma$.
Since $\mathrm{tr} ( S) = 0$, we find that $S$ has spectrum
\[
\left\{ \left[-\sqrt{\dfrac{d(n-1)}{n-d}}\right]^{n-d},\left[\sqrt{\dfrac{(n-1)(n-d)}{d}}\right]^{d} \right\}.
\]
\end{proof}

When the equality occurs, the seidel matrix has two distinct eigenvalues, which correpsond to the equiangular line system in $\mathbb R^d$ with $n$ points and common angle $\sqrt{\frac{n-d}{d(n-1)}}$. 

\begin{theorem}[\cite{LEMMENS1973494}] \label{thm:SphericalOdd}
Let $X$ be an $n$ points spherical two-distance set in $\mathbb S^{d-1}$ with inner products $a,b$. If $n>2d+2$ and $a+b \ge 0$, $(a+b-2)/(b-a)$ is an odd integer.
\end{theorem} 

    Larman, Rogers, and Seidel \cite{larman1977two} claimed that the distance square $\delta^2$ equals to $(k-1)/k$ for some integer $k$, which is the same result if we let $-(a+b-2)/(b-a) = 2k-1$.
    
\begin{proof}

    The proof follows similarly to that of Theorem \ref{thm:odd}.
\end{proof}

\begin{definition}
    For an odd integer $\gamma$, define $M^+_\gamma(d)$ as the maximum cardinality of spherical two-distance sets in $\mathbb S^{d-1}$ such that $-(a+b-2)/(b-a) = \gamma$ and $a+b\ge0$.
\end{definition}


\begin{theorem}[\cite{LEMMENS1973494}] Suppose $m, d$ are positive integer with $(2m+1)^2 > d > 3$. Then
\[
M^+(d) \le \max\left\{M^+_3(d), M^+_5(d),\dots, M^+_{2m+1}(d) , \frac{4dm(m+1)}{4m^2+4m-d}\right\}.
\]
\end{theorem}
\begin{proof}
    By Theorem \ref{thm:SphericalOdd}, $-(a+b-2)/(b-a) = \gamma$ is an odd integer. For a positive number $m$, if $\gamma > 2m+1$, since $\gamma > (2m+1)^2 > d$, by Theorem \ref{lem:SpericalRelativeBound}, we have
    \begin{equation*}
    \begin{aligned}
        M^+_{\gamma}(d) \le \frac{d\left(\gamma^2 - 1\right)}{\gamma^2-d} &= d +\frac{d(d-1)}{\gamma^2-d}\\
    &\le d +\frac{d(d-1)}{(2m+1)^2-d} = \frac{4dm(m+1)}{(2m+1)^2-d}.
    \end{aligned}
    \end{equation*}
\end{proof}

\newpage
\begin{table}[!ht]
    \centering
    \begin{tabular}{c|cccc|c}
         \diagbox{dimension}{upper bound}{ $\ k$} & 2 & 3 & 4 & 5 & $M^+(d)$\\  \hline     
         $5$ & $10$ & 6  & 5  & 5  & 16 \\
          6  & $16$ & 7 & 6  & 6  & 27 \\  
         $7$ & $28$ & 9 & 8 & 7  & 28 \\ 
         $8$ & $64$ & $11$  & 9 & 8 & 36\\ \hline
          9   & & $13$  &  10 & 10 & 45 \\
         $10$ & & $16$  &  12 & 11 & 55 \\
         $11$ & & $18$  &  13 & 12 & 66 \\
         $12$ & & $22$  &  15 & 13 & 78 \\
         $13$ & & $26$  &  17 & 15 & 91 \\
         $14$ & & $30$  &  19 & 16 & 105 \\
         $15$ & & $36$  &  21 & 18 & 120 \\
         $16$ & & $42$  &  23 & 19 & 136 \\
         $17$ & & $51$  &  25 & 21 & 153 \\
         $18$ & & $61$  &  27 & 22 & 171 \\
         $19$ & & $76$  &  30 & 24 & 190 \\
         $20$ & & $96$  &  33 & 26 & 210 \\
         $21$ & & $126$ &  36 & 28 & 231 \\
         $22$ & & $176$ &  39 & 29 & 275 \\
         $23$ & & $276$ &  42 & 31 & 276 \\ 
         $24$ & & $576$ &  46 & 33 & 300 \\ \hline
          25  & & & $50$  & 35 & 325 \\ 
         $26$ & & & $54$  & 37 & 351 \\
         $27$ & & & $58$  & 40 & 378 \\    
         $28$ & & & $64$  & 45 & 406 \\
         $29$ & & & $69$  & 48 & 435 \\
         $30$ & & & $75$  & 50 & 465 \\
         $31$ & & & $82$  & 53 & 496 \\
         $32$ & & & $90$  & 56 & 528 \\
         $33$ & & & $99$  & 58 & 561 \\
    \end{tabular}
    \caption{Upper bounds on spherical two-distance sets in $d$-dimensional Euclidean space with $a + b \geq 0$, as given by Theorem \ref{lem:SpericalRelativeBound}. For each dimension $d$, the upper bound applies for the L.R.S. constant $k = 2, 3, 4, 5$, where $-(a+b-2)/(b-a) = 2k-1$.}
    \label{tab:boundsSpherical}
\end{table}
\begin{example}
In Table \ref{tab:boundsSpherical}, for $\mathbb{R}^{15}$, the upper bound is 36 points for $-\frac{a+b-2}{b-a} \ge 2\times 3-1 = 5$. However, we have constructed a configuration with 120 points for $a = 0, b = \frac 12$ (the mid-point case). 
In the case, $-\frac{a+b-2}{b-a} = 3$. This indicates that the maximum two-distance set is achievable only for the specific inner products with $-\frac{a+b-2}{b-a} = 3$.
\end{example}

\section*{6. Upper bound of spherical two-distance sets when $a+b < 0$}
\stepcounter{counter}

\begin{lemma}\label{lem:Seidel}
    There exists a $n\times n$ Seidel matrix has smallest eigenvalue $\lambda_0 < -1$ with multiplicity $n-d$ if and only if there exists a equiangular line system with $n$ points in $\mathbb R^{d}$.
\end{lemma}

\begin{proof}
    If there exists an $n$ points equiangular line system in $\mathbb S^{d-1}$ with common angle $\alpha < 1$, let the Gram matrix $G$ of it is positive semidefinite with rank $d$. We define the Seidel matrix $S$ with 
    \[
    S = \frac{G - I}{\alpha}.
    \]
    It is easy to check that $S$ has smallest eigenvalue $1/\alpha < -1$ with multiplicity $n-d$.

    Suppose there exists a $n\times n$ Seidel matrix has smallest eigenvalue $\lambda_0 < -1$ with multiplicity $n-d$, and denote it as $S$. Define
    \[
    G := -\frac{S}{\lambda_0} + I.
    \]
    
    Let $S$ has eigenvalues $\lambda_0 \leq \lambda_1 \leq \dots \leq \lambda_{n}$ with corresponding eigenvectors $v_i$. Then
    \[
    G \cdot v_i = \left(-\frac{S}{\lambda_0} + I\right)v_i = \left(-\frac{\lambda_i}{\lambda_0} + 1\right)v_i.
    \]
    
    Thus $-\frac{\lambda_i}{\lambda_0} + 1$ are the eigenvalues of $G$, and since $\lambda_0 < 0$, we have
    \[
    0 = -\frac{\lambda_0}{\lambda_0} + 1 \leq -\frac{\lambda_1}{\lambda_0} + 1 \leq \dots \leq -\frac{\lambda_{n}}{\lambda_0} + 1.
    \]
    
    Therefore, all eigenvalues of $G$ are non-negative, making $G$ positive semidefinite. Moreover, the smallest eigenvalue of $G$ is $0$ with multiplicity $n-d$. By the definition of Seidel matrices, $G$ is a real symmetric matrix.
    
    For any real symmetric positive semidefinite matrix $G$, we can orthogonally diagonalize it as $G = Q\Lambda Q^T$, where $Q^TQ = I$ and
    \[\Lambda = \begin{bmatrix}
        x_1 & 0    & \dots & 0\\
        0   & x_2  & \dots & 0\\
        \vdots   & \vdots  & \ddots & \vdots\\
        0   & 0  & \dots & x_n\\
    \end{bmatrix},\]
    where $x_i$ are the eigenvalues.
    
    Since $G$ is positive semidefinite, $x_i \geq 0$ for all $i$, so we can define
    \[\Lambda^{1/2} = \begin{bmatrix}
        \sqrt{x_1} & 0    & \dots & 0\\
        0   & \sqrt{x_2}  & \dots & 0\\
        \vdots   & \vdots  & \ddots & \vdots\\
        0   & 0  & \dots & \sqrt{x_n}\\
    \end{bmatrix}.\]
    
    Setting $A = \Lambda^{1/2}Q^T$, we obtain
    \[
    G = Q\Lambda Q^T = (Q\Lambda^{1/2})(\Lambda^{1/2}Q^T) = (\Lambda^{1/2}Q^T)^T(\Lambda^{1/2}Q^T) = A^TA.
    \]
    
    Therefore, $G$ is a Gram matrix, and $A$ is an $n \times n$ matrix corresponding to a set of equiangular lines with inner product $\frac{-1}{\lambda_0} < 1$ and $\text{rank}(A) = d$. This implies that $A$ represents a set of $n$ equiangular lines in $\mathbb{R}^d$.
\end{proof}

\begin{theorem}
Let $X$ be an $n$ points spherical two-distance set in $\mathbb S^{d-1}$ with inner products $a,b$.  If $a+b < 0$, then 
\[
n \le N(d+1)
\]
    where $N(d)$ denotes the maximum number of equiangular lines in $\mathbb{R}^d$.
\end{theorem}

The theorem was proven in \cite{delsarte1991spherical}. Now we give a new proof.

\begin{proof} If $a+b < 0$, by Theorem \ref{thm:eigenvalues}, Seidel matrix $S = \frac{2G - (a+b)\cdot J  + (a+b-2) \cdot I } {b-a}$ has the smallest eigenvalue $\frac{a+b-2}{b-a}$ with the multiplicity at least $n-d-1$.

Since 
\[
\frac{a+b-2}{b-a} + 1 =  \frac{2b-2}{b-a} = \frac{2(b-1)}{b-a} < 0,
\]
by lemma \ref{lem:Seidel}, there exists a equiangular line system with $n$ points in $\mathbb R^{d+1}$.

Therefore, by the definition of function $N$, we obtain 
\[
n \le N(d+1).
\]
\end{proof}

\begin{theorem} \label{correspondence}
There exists a bijective correspondence between $n$ points equiangular line systems in $\mathbb{R}^{d+1}$ with common angle $\alpha$ and families of $n$ points spherical two-distance sets in $\mathbb{S}^{d-1}$ with inner products $a, b$ satisfying $a+b < 0$ and $\frac{a-b}{a+b-2} = \alpha$.
\end{theorem}
\begin{proof}
We prove both directions of the correspondence.

\textbf{Step 1:} From spherical two-distance set to equiangular line system.

Let $X$ be a $n$ point spherical two-distance set in $\mathbb{S}^{d-1}$ with inner products $a, b$ where $a+b < 0$. Let $G = G(X)$ be the Gram matrix of $X$. Define the Seidel matrix $S$ by
\[
S = \frac{2G - (a+b) \cdot J + (a+b-2) \cdot I}{b-a}.
\]
By Theorem \ref{thm:eigenvalues}, the Seidel matrix $S$ has smallest eigenvalue $\frac{a+b-2}{b-a}$ with multiplicity at least $n-d-1$.

We construct a new matrix $G'$ by
\[
G' = -\frac{S}{\frac{a+b-2}{b-a}} + I = \frac{a-b}{a+b-2} \cdot S +I.
\]

\textbf{Claim:} $G'$ is the Gram matrix of an equiangular line system in $\mathbb{R}^{d+1}$.

Since $-\frac{a+b-2}{b-a} > 0$, the matrix $G'$ has smallest eigenvalue $0$ with multiplicity at least $n-d-1$, which implies that $G'$ is positive semidefinite. Since $S$ is symmetric and the diagonal entries of $G'$ are all 1's, $G'$ is a valid Gram matrix.

The non-diagonal entries of $S$ are either $1$ or $-1$, so the non-diagonal entries of $G'$ are either $\frac{b-a}{2-a-b}$ or $-\frac{b-a}{2-a-b}$. This shows that $G'$ corresponds to an equiangular line system with common angle $\alpha = \frac{b-a}{2-a-b}$.

Since $\text{rank}(G') \le d+1$, the equiangular line system lies in $\mathbb{R}^{d+1}$.

\textbf{Step 2:} From equiangular line system to spherical two-distance set.

Conversely, let $Y$ be a $n$ points equiangular line system in $\mathbb{R}^{d+1}$ with common angle $\alpha$. Let $G'$ be its Gram matrix. Define the Seidel matrix
\[
S = \frac{a+b-2}{a-b} \cdot (G' -I),
\]
where $a, b$ are chosen such that $\frac{b-a}{2-a-b} = \alpha$ and $a+b < 0$. 

\textbf{Claim.} $S$ is a Seidel matrix.
The diagonal entries of $G'$ are $1$, so the diagonal entries of $S$ are $\frac{a+b-2}{a-b} \cdot(1-1)=0$. The non-diagonal entries of $G'$ are either $\alpha$ or $-\alpha$. Since $\frac{b-a}{2-a-b} = \alpha$, the non-diagonal entries of $S$ are $\frac{1}{\alpha} \cdot(\pm\alpha-0)=\pm1$. 

Therefore, $S$ has zero diagonal entries and non-diagonal entries equal to $\pm 1$, confirming that $S$ is indeed a Seidel matrix.

Then we can construct the Gram matrix
\[
G = \frac{(b-a)\cdot S - (a+b-2)\cdot I + (a+b)\cdot J}{2},
\]
which corresponds to a spherical two-distance set in $\mathbb{S}^{d-1}$ with inner products $a, b$.

Given $\alpha$, the constraint $\frac{a-b}{a+b-2} = \alpha$ with $a+b < 0$ can simplify to 
\[
b = \frac{(1-\alpha)a +2\alpha}{1+\alpha}.
\]

Therefore, each equiangular line system corresponds to a family of spherical two-distance sets.

\end{proof}

\begin{corollary}
If there exists a $n$ point equiangular line system in $\mathbb{R}^{d+1}$ with common angle $\alpha$, then we can construct spherical two-distance sets in $\mathbb{S}^{d-1}$ with inner products $a, b$ satisfying $\frac{a-b}{a+b-2} = \alpha$ and $a+b < 0$.
\end{corollary}

\begin{example}
Since there exists a 36 points equiangular line system in $\mathbb{R}^{15}$ with common angle $1/5$, we define the Gram matrix 
\[
G = \frac{S \cdot (b-a) - (a+b-2)\cdot I + (a+b)\cdot J}{2},
\]
where $S = S(X)$ is the Seidel matrix of the equiangular line system and $a, b$ satisfy
\[
 \frac{a-b}{a+b-2} = \frac{1}{5} 
\]
\[
\Rightarrow b = \frac{2a+1}{3}
\]
The Gram matrix $G$ corresponds to a family of spherical two-distance sets in $\mathbb{S}^{13}$ with inner products $a, b$ satisfy $2a + 1 = 3b$ and $a+
b<0$. 
\end{example}

\begin{corollary} \label{small}
 Let $X$ be an $n$ points spherical two-distance set in $\mathbb S^{d-1}$ with inner products $a,b$.  If $a+b < 0$, then 
    \[
    n \le \dfrac{(d+1)\left(\left(\dfrac{a+b-2}{b-a}\right)^2-1\right)}{\left(\dfrac{a+b-2}{b-a}\right)^2-(d+1)}
    \]
\end{corollary}

\begin{proof}
    According to Theorem \ref{correspondence}, $X$ corresponds to a $n$ point equiangular line system in $\mathbb{R}^{d+1}$ with common angle $\alpha$, where $\frac{a-b}{a+b-2} = \alpha$.
    In the study by Lemmens and Seidel \cite{LEMMENS1973494}, we have
    
    \[
    n \le \dfrac{(d+1)\left(1- \alpha^2\right)}{1-(d+1)\alpha^2}.
    \]
    Therefore, we obtain a similar inequality for spherical two-distance sets:
    \[
    n \le \dfrac{(d+1)\left(\left(\dfrac{a+b-2}{b-a}\right)^2-1\right)}{\left(\dfrac{a+b-2}{b-a}\right)^2-(d+1)}.
    \]
    
\end{proof}

\section*{7. Equiangular Tight frames}
\stepcounter{counter}

In this section, we focus theorem \ref{lem:RelativeBound}, and theorem \ref{lem:SpericalRelativeBound} when equalities hold.
For example, if 

\[
    \dfrac{(d+1)\left(\left(\dfrac{1+\delta^2}{1-\delta^2}\right)^2 - 1\right)}{\left(\dfrac{1+\delta^2}{1-\delta^2}\right)^2-(d+1)}+1 \in \mathbb Z,
\]
the Seidel matrix has spectrum 
\[
\left\{ \left[-\sqrt{\dfrac{(n-2)(d+1)}{n-d-2}}\right]^{n-d-2},\left[\sqrt{\dfrac{(n-2)(n-d-2)}{d+1}}\right]^{d+1} \right\}.
\]

\begin{definition} [\cite{holmes2004optimal}]
    Let H be a Hilbert space, real or complex, and let $F = \{f_i\}_{i\in \mathbb I} \subset H$ be a subset. We call $F$ a frame for $H$ provided that there are two constants $C, D > 0$ such that
the inequality 
\[C\|x\|^2 \le
\|\langle x,f\rangle\|^2 \le D\|x\|^2, 
j\in I
\]
holds for every $x \in H$. When $C = D$, then we call $F$ a tight frame.
\end{definition}

Holmes and Paulsen \cite{holmes2004optimal} showed that a Seidel matrix with exactly two distinct eigenvalues corresponds to an Equiangular Tight Frame (ETF). 

\begin{theorem}[Theorem 3.3 of \cite{holmes2004optimal}] \label{thm:3.3}
 Let $Q$ be a self-adjoint $n \times n$ matrix $Q$ with $q_{i,i} = 0$ for all $i$ and
$|q_{i,j}| = 1$ for all $i \ne j$. Then the following are equivalent: 
\begin{enumerate}[(i)]
    \item $Q$ is the signature matrix of a $2$-uniform $(n, d)$-frame,
    \item $Q^2 =(n-1)I + \mu Q$
for some necessarily real number $\mu$,
    \item  $Q$ has exactly two eigenvalues, $\rho_1 \ge \rho_2$.
\end{enumerate}
\end{theorem}

 \textit{(i)} is equivalence to "$Q$ is a Seidel matrix of an $n$ points ETF in $\mathbb R^d$" in our case. Then we have the following theorem:
\begin{theorem} \label{ETFs}
If there exists a $n$ point two-distance set in $\mathbb{R}^d$ with distances $1, \delta$ where
\[
    n = \dfrac{(d+1)\left(\left(\dfrac{1+\delta^2}{1-\delta^2}\right)^2 - 1\right)}{\left(\dfrac{1+\delta^2}{1-\delta^2}\right)^2-(d+1)}+1 \in \mathbb{Z},
\] 
then there exists an equiangular tight frame with $n-1$ vectors in $\mathbb{R}^{d+1}$.
\end{theorem}

\begin{proof}
    If $X$ is a two-distance set satisfying the above condition, then by Theorem \ref{lem:RelativeBound}, the Seidel matrix $S = S(X)$ has spectrum
    \[
\left\{ \left[-\sqrt{\dfrac{(n-2)(d+1)}{n-d-2}}\right]^{n-d-2},\left[\sqrt{\dfrac{(n-2)(n-d-2)}{d+1}}\right]^{d+1} \right\}.
\]
According to Theorem \ref{thm:3.3}, $S$ is a Seidel matrix which corresponds to an equiangular tight frame with $n-1$ vectors in $\mathbb{R}^{d+1}$.
\end{proof}

Theorem \ref{ETFs} implies that given $d, n, \delta$ such that $n =\frac{(d+1)\left(\left(\frac{1+\delta^2}{1-\delta^2}\right)^2 - 1\right)}{\left(\frac{1+\delta^2}{1-\delta^2}\right)^2-(d+1)} \in \mathbb Z$. 
If there does not exist an equiangular tight frame having $\frac{(d+1)\left(\left(\frac{1+\delta^2}{1-\delta^2}\right)^2 - 1\right)}{\left(\frac{1+\delta^2}{1-\delta^2}\right)^2-(d+1)}$ vectors in $\mathbb{R}^{d+1}$, then 
\[
g(d) < \dfrac{(d+1)\left(\left(\dfrac{1+\delta^2}{1-\delta^2}\right)^2 - 1\right)}{\left(\dfrac{1+\delta^2}{1-\delta^2}\right)^2-(d+1)} +1.
\]

\begin{example} \label{ex1}
    According to \cite{fickus2015tables}, there does not exist an equiangular tight frame with 76 vectors in $\mathbb{R}^{19}$. By Theorem \ref{ETFs}, this implies that $k = 2, 3, \dots $
    \[
    g_{2k-1}(18) < 77.
    \]
This improves the bound in Table \ref{tab:bounds}.
\end{example}

\begin{theorem} \label{ETFsSpherical}
There exists a $n$ points spherical two-distance set in $\mathbb{S}^{d-1}$ with inner products $a, b$ (where $a+b \geq 0$) if there exists an Equiangular Tight Frame with $n$ vectors in $\mathbb{R}^d$, where
\[
    n = \frac{d\left(\left(\frac{a+b-2}{b-a}\right)^2-1\right)}{\left(\frac{a+b-2}{b-a}\right)^2-d} \in \mathbb{Z}.
\]  
\end{theorem}

\begin{proof}
    If $X$ is a two-distance set satisfying the above condition, then by Theorem \ref{lem:SpericalRelativeBound}, the Seidel matrix $S = S(X)$ has spectrum
 \[
\left\{ \left[-\sqrt{\dfrac{d(n-1)}{n-d}}\right]^{n-d},\left[\sqrt{\dfrac{(n-1)(n-d)}{d}}\right]^{d} \right\}.
\]
According to Theorem \ref{thm:3.3}, $S$ is a Seidel matrix which corresponds to an equiangular tight frame with $n$ vectors in $\mathbb{R}^{d}$.
\end{proof}

\begin{example} \label{ex2}
    According to \cite{fickus2015tables}, there does not exist an equiangular tight frame with 76 vectors in $\mathbb{R}^{19}$. By Theorem \ref{ETFsSpherical}, this implies that $k = 2, 3, \dots $
    \[
    M_{2k-1}^{+}(19) < 76.
    \]
This improves the bound in Table \ref{tab:boundsSpherical}.

    By Theorem \ref{correspondence}, this implies that $k = 2, 3, \dots $
    \[
     M_{2k-1}^{-}(18) < 76.
    \]
This improves the bound in Theorem \ref{small}.
\end{example}
\newpage
\begin{table}[!ht]
    \centering
    \begin{tabular}{c|ccc}
         $d$ & $g_{5}(d)$ & $M_{5}^+(d)$ & $M_{5}^-(d)$  \\ \hline
          9  & $17$  &  $13$  & $16$  \\
          10 & $19$  &  $16$  & $19$  \\
          11 & $23$  &  $18$  & $23$  \\
          12 & $27$  &  $22$  & $26$  \\
          13 & $31$  &  $26$  & $31$  \\
          14 & $37$  &  $30$  & $36$  \\
          15 & $43$  &  $36$  & $43$  \\
          16 & $52$  &  $42$  & $51$  \\
          17 & $62$  &  $51$  & $62$  \\
          18 & {\color{red}76}  &  $61$  & {\color{red}75}  \\
          19 & $97$  &  {\color{red}75}  & $96$  \\
          20 & $127$ &  $96$  & $126$ \\
          21 & $177$ &  $126$ & $176$ \\
          22 & $277$ &  $176$ & $276$ \\
          23 &  577  &  $276$ & $576$ \\
    \end{tabular}
    \caption{The comparison between two-distance sets and spherical two-distance sets is presented, along with the upper bounds established by Theorem \ref{lem:RelativeBound}, Theorem \ref{lem:SpericalRelativeBound}, and Corollary \ref{small}. Corrections highlighted in red are provided in Example \ref{ex1} and Example \ref{ex2}.}
    \label{tab:compare}
\end{table}

In this paper, we analyze the spectrum of Seidel matrices constructed from distance relationships to establish upper bounds for two-distance sets in Euclidean space and on the unit sphere. Our key contribution is developing a transformation from Cayley-Menger matrices (for Euclidean sets) and Gram matrices (for spherical sets) to Seidel matrices, extending the classical Lemmens-Seidel approach \cite{LEMMENS1973494} to the two-distance setting. Moreover, we establish the bijective correspondence between equiangular line systems and spherical two-distance sets with $a+b < 0$, and reveal the connection between the existence of ETFs (Equiangular Tight Frames) and the attainability of our upper bounds.

Our upper bounds fundamentally arise from eigenvalue multiplicity constraints that prevent the Seidel matrix structure from becoming degenerate. For a Seidel matrix of order $n$ associated with a $d$-dimensional point configuration, the rank constraints force certain eigenvalues to have bounded multiplicities. When these multiplicities exceed their geometric limits, the corresponding point configuration becomes impossible to realize in the given dimension.

In future work, we plan to investigate Seidel matrices with a smallest eigenvalue of multiplicity $1$ and a second smallest eigenvalue of multiplicity $n-d-3$, which contrasts with equiangular line systems where the smallest eigenvalue has multiplicity $n-d$. By analyzing the possible eigenvalue distributions and adapting methods from equiangular line theory, we aim to determine the existence of two distance set or spherical two distance set with $a+b\ge 0$.

\bibliographystyle{abbrv}
\bibliography{ref}

\end{document}